\documentclass[11pt,twoside,a4paper]{article}
\usepackage{amsmath, amsthm, amssymb, enumerate}
\usepackage[pdftex]{graphicx}
\graphicspath{ {Images/} }

\begin{document}

\title{Saturated Graphs of Prescribed Minimum Degree}
\date{}
\author{A. Nicholas Day \thanks{School of Mathematical Sciences, Queen Mary University of London, London E1 4NS, UK}}
\date{\today}
\maketitle 
\newtheorem{Thm}{Theorem}
\newtheorem{Lemma}[Thm]{Lemma}
\newtheorem{Def}[Thm]{Definition}
\newtheorem{Coro}[Thm]{Corollary}
\newtheorem{Conj}[Thm]{Conjecture}
\newtheorem{Prop}[Thm]{Proposition}
\newtheorem{Quest}[Thm]{Question}
\newtheorem{Claim}{Claim}
\newtheorem{Case}{Case}

\begin{abstract}
A graph $G$ is $H$-saturated if it contains no copy of $H$ as a subgraph but the addition of any new edge to $G$ creates a copy of $H$.  In this paper we are interested in the function sat$_{t}(n,p)$, defined to be the minimum number of edges that a $K_{p}$-saturated graph on $n$ vertices  can have if it has minimum degree at least $t$.  We prove that sat$_{t}(n,p) = tn - O(1)$, where the limit is taken as $n$ tends to infinity.  This confirms a conjecture of Bollob\'{a}s when $p = 3$.  We also present constructions for graphs that give new upper bounds for sat$_{t}(n,p)$.
\end{abstract}

\textit{2010} Mathematics subject classification:  	05C35

\section{Introduction}

We say a graph $G$ is $H$-saturated if it contains no copy of $H$ as a subgraph but the addition of any new edge to $G$ creates a copy of $H$.  In this paper we are interested in the case where $H$ is the complete graph on $p$ vertices, denoted $K_{p}$.  For further results  on saturated graphs see surveys by either Faudree, Faudree and Schmitt \cite{Faudree} or Pikhurko \cite{Oleg}.  Erd\H{o}s, Hajnal and Moon \cite{EHM} showed that if $G$ is a $K_{p}$-saturated graph on $n$ vertices then $e(G) \geqslant n(p-2) - \binom{p-1}{2}$ and that the unique graph achieving equality is formed by taking a clique on $p-2$ vertices and fully connecting it to an independent set of size $n - (p - 2)$.  This extremal graph has minimum degree $p-2$ and no $K_{p}$-saturated graph on at least $p$ vertices can have smaller minimum degree.  Thus it is natural to ask: how few edges can a $K_{p}$-saturated graph have if it has minimum degree at least $t$ for $t \geqslant p-2$?  

Observe that any $K_{3}$-saturated graph on $n$ vertices must be connected and so cannot have fewer than $n-1$ edges.  The graph comprising of a single vertex connected to all other vertices is $K_{3}$-saturated, has minimum degree $1$ and has this minimum number of edges.  Duffus and Hanson \cite{DuffusHanson} showed that any $K_{3}$-saturated graph on $n$ vertices with minimum degree $2$ has at least $2n-5$ edges.  Moreover, they showed that the unique graphs achieving this are obtained by taking a $5$-cycle and repeatedly duplicating vertices of degree $2$, that is, picking a vertex of degree $2$ and adding a new vertex to the graph with the same neighbourhood as the chosen vertex.  They also showed that any $K_{3}$-saturated graph on $n \geqslant 10$ vertices with minimum degree $3$ has at least $3n - 15$ edges and that any graph achieving this contains the Petersen graph as a subgraph.

In this paper we consider the function
\begin{equation}
\text{sat}_{t}(n,p) = \min \{ e(G): |V(G)| = n, G \text{ is } K_{p} \text{-saturated} , \delta(G) \geqslant t \},\nonumber
\end{equation}
where $\delta(G)$ is the minimum degree of $G$.  We also define the set Sat$_{t}(n,p)$ to be
\begin{equation}
 \{ G : |V(G)| = n, G \text{ is } K_{p} \text{-saturated} , \delta(G) \geqslant t, e(G) = \text{sat}_{t}(n,p) \}.\nonumber
\end{equation}
The complete bipartite graph $K_{t,n-t}$ shows that for $n \geqslant 2t$ we have sat$_{t}(n,3) \leqslant tn - t^{2}$.  This upper bound and Duffus and Hanson's results led Bollob\'{a}s \cite{Boll} to conjecture that for fixed $t$ we have sat$_{t}(n,3) = tn - O(1)$.

For more general values of $p$, Duffus and Hanson \cite{DuffusHanson} showed that sat$_{t}(n,p) \geqslant \frac{t + p -2}{2}n  - O(1)$.  Writing $\alpha(G)$ for the size of the largest independent set in $G$, Alon, Erd\H{o}s, Holzman and Krivelevich \cite{AEHK}  showed that any $K_{p}$-saturated graph on $n$ vertices with at most $O(n)$ edges has $\alpha(G) \geqslant n - O(\frac{n}{\log \log n})$.  This shows that sat$_{t}(n,p) \geqslant tn - O(\frac{n}{\log \log n})$ as $e(G) \geqslant \alpha(G)\delta(G)$.  Pikhurko \cite{Oleg} improved this result to show that sat$_{t}(n,p) \geqslant tn - O(\frac{n \log \log n}{\log n})$.

Our main result in this paper improves these results by confirming and generalising Bollob\'{a}s's conjecture. 

\begin{Thm}\label{Thm1}
Let $t \in \mathbb{N}$. There exists a constant $c = c(t)$ such that, for all $3 \leqslant p \in \mathbb{N}$ and all $n \in \mathbb{N}$, if $G$ is a $K_{p}$-saturated graph of order $n$ and minimum degree at least $t$ then $e(G) \geqslant tn - c$.
\end{Thm}

The proof of Theorem \ref{Thm1} is presented in Section 2.  To see that this result is best possible (up to the value of the constant) consider the graph obtained from fully connecting a clique of size $p-3$ to the complete bipartite graph $K_{t-(p-3),n-t}$.  This graph is $K_{p}$-saturated and has minimum degree $t$, showing that 
\begin{eqnarray}\label{eq1}
\text{sat}_{t}(n,p) \leqslant tn - t^{2} + t(p-3)- \binom{p-2}{2} 
\end{eqnarray}
for $n \geqslant 2t - (p-3)$ and $ t \geqslant p-2$.

We remark that though it may seem surprising that the constant $c(t)$ in the statement of Theorem \ref{Thm1} doesn't depend on $p$, it is a consequence of the fact that any $K_{p}$-saturated graph (on at least $p -1$ vertices) has minimum degree at least $p-2$.  As a result, Theorem \ref{Thm1} is trivially true whenever $p \geqslant 2t + 2$, and so, for fixed $t$, there are only a finite number of values of $p$ we need to consider.  The independence of $c(t)$ from $p$ is also reflected in our proof of Theorem \ref{Thm1} which only makes use of the fact that our graph is $K_{p}$-saturated for some $3 \leqslant p \in \mathbb{N}$ and doesn't make use of $p$'s value in any way.

On the other hand, Theorem \ref{Thm1} can be used to show the following:  for all $t,p \in \mathbb{N}$ with $ t \geqslant p-2 \geqslant 1$, there exists a constant $c(t,p)$ such that, for all sufficiently large $n \in \mathbb{N}$, we have sat$_{t}(n,p) = tn - c(t,p)$. Indeed, Theorem \ref{Thm1} together with (\ref{eq1}) shows that, for $n$ sufficiently large, all $G \in$ Sat$_{t}(n,p)$ have $\delta(G) = t$.  Duplicating a vertex of degree $t$ in such a graph $G$ gives a $K_{p}$-saturated graph on $n+1$ vertices with minimum degree $t$ and sat$_{t}(n,p) + t$ edges.  Thus, as $n$ increases, the integer sequence $tn - $sat$_{t}(n,p)$ becomes non-decreasing but bounded above by $c(t)$ and so is eventually constant.
  
The proof of Theorem \ref{Thm1} can be used to show that $c(t,p) \leqslant t^{(t^{(2t^{2})})}$.  In Section 3 we discuss constructing $K_{p}$-saturated graphs and prove a lower bound for $c(t,p)$.

\begin{Thm}\label{Thm2}
Let $3 \leqslant p \in \mathbb{N}$.  There exists a constant $C = C(p) > 0$ such that, for all sufficiently large $t \in \mathbb{N}$, we have $c(t,p) \geqslant C 2^{t}t^{3/2}$.
\end{Thm}

The large distance between these upper and lower bounds for $c(t,p)$ naturally leads to the problem of improving these bounds, or perhaps even determining $c(t,p)$ for all $t$ and $p$.  Our proof of Theorem \ref{Thm1} seems to be inefficient for the purposes of bounding $c(t,p)$ and so we believe $c(t,p)$ is likely to be closer to the lower bound we give in Theorem \ref{Thm2} than the upper bound obtained from Theorem \ref{Thm1}.

We remark that one may also ask how few edges a $K_{p}$-saturated graph can have if restrictions are placed on its maximum degree rather than its minimum degree.  Results on this problem for $p = 3$ can be found in the paper of F\"{u}redi and Seress \cite{FurSer} and also in the paper of Erd\H{o}s and Holzman \cite{EHol}.  Results for the case $p = 4$ can be found in the paper of Alon, Erd\H{o}s, Holzman and Krivelevich \cite{AEHK}.  There are currently no known results for $p \geqslant 5$.

\section{Proof of Theorem \ref{Thm1}}

For a graph $G$ and a vertex $v \in V(G)$, let $N(v)$ be the set of vertices in $G$ that are adjacent to $v$.  For $X \subseteq V(G)$ let $N_{X}(v) = N(v) \cap X$, let $d_{X}(v) = |N_{X}(v)|$ and let $e(X)$ be the number of edges in the graph $G[X]$.  For another set $Y \subseteq V(G)$ that is disjoint from $X$, let $N_{Y}(X)$ be the set of vertices in $Y$ adjacent to $X$ and let $e(X,Y)$ be the number of edges between $X$ and $Y$.

\begin{proof}[Proof of Theorem \ref{Thm1}]

Let $G$ be a $K_{p}$-saturated graph on vertex set $V$ with $|V| = n$ and $\delta(G) \geqslant t$.  Given a set $R \subseteq V$, let $\overline{R}$ be the closure of $R$ under $t$-neighbour bootstrap percolation on $G$.  That is, let $\overline{R} = \bigcup_{i \geqslant 0} R_{i}$ where $R_{0} = R$ and
\begin{equation}
R_{i} = R_{i-1} \cup \{ v \in V : d_{R_{i-1}}(v) \geqslant t \} \nonumber
\end{equation}
for $i \geqslant 1$.  Any vertex $x \in R_{i} \setminus R_{i-1}$ sends at least $t$ edges to $R_{i-1}$ and so $e(\overline{R}) \geqslant t(|\overline{R}| - |R|)$.  Let $Y(R) = V \setminus \overline{R}$ and for a vertex $v \in V$ let
\begin{equation}
w_{R}(v) = d_{\overline{R}}(v)  + \frac{1}{2}d_{Y(R)}(v). \nonumber
\end{equation}
We call $w_{R}(v)$ the weight of $v$ (with respect to $R$).  Within $Y(R)$, we define $B(R)$ to be the set $\{ v \in Y(R): w_{R}(v) < t \}$, which we call the set of bad vertices.  Our aim will be to prove the following claim.
\begin{Claim}\label{Claim1} 
There exists a constant $c_{1} = c_{1}(t)$ and a set $R \subseteq V$ with $|R| \leqslant c_{1}(t)$ such that $B(R) = \emptyset$.
\end{Claim}
If we can prove Claim \ref{Claim1} then we have proved the theorem as
\begin{eqnarray}
e(G) &=& e(\overline{R}) + e(\overline{R},Y(R)) + e(Y(R))  \nonumber \\
& \geqslant & t(|\overline{R}|-|R|)  + \sum_{y \in Y(R)} w_{R}(y) \nonumber \\
& \geqslant & t(|\overline{R}| - c_{1}) + t|Y(R)| \nonumber \\
& = & t(n - c_{1}),\nonumber
\end{eqnarray}
as required.
To prove Claim \ref{Claim1}, we would like to show that if a set $R \subseteq V$ does lead to $B(R)$ being non-empty then we can move a small number of vertices into $R$ so that the remaining vertices in $B(R)$ have strictly larger weight.  If so, we can start with some initial small set of vertices $R$ and keep moving small numbers of vertices into $R$ until $B(R)$ is empty.  This idea of moving vertices into $R$ fits naturally with our set up so far.  Indeed, suppose that $S$ is a set of vertices with $R \subseteq S$.  We have that $\overline{R} \subseteq \overline{S}$ and $Y(R) \supseteq Y(S)$ and so $w_{R}(v) \leqslant w_{S}(v)$ for all $v \in V$.  Thus, we have that $B(R) \supseteq B(S)$.

It turns out that dealing with $w_{R}(v)$ directly is difficult and so we introduce a control function $l_{R}(v) = \sum_{x \in N(v)}f_{R}(x)$ defined for all $v \in V$, where for all $x \in V$
\begin{equation}
  f_{R}(x)=\begin{cases}
    1, & \text{if } x \in R, \\
    \frac{1}{2}, & \text{if } x \in \overline{R}\setminus R, \\
    \frac{1}{2t}d_{R}(x), & \text{if } x \in Y(R). \nonumber
  \end{cases}
\end{equation}
Observe that $l_{R}(v) \leqslant w_{R}(v)$ for every $v \in V$, since $d_{R}(X) \leqslant t-1$ for every $x \in Y(R)$.  Similarly, we have $f_{R}(v) \leqslant f_{S}(v)$ for every $R \subseteq S$ and every $v \in V$, since $Y(S) \subseteq Y(R)$.
  
We use our control function $l_{R}(v)$ to make the following claim.

\begin{Claim}\label{Claim2}
For every set $R \subseteq V$, there exists a set $S \subseteq V$ such that $R \subseteq S $, $|S| \leqslant |R| + t2^{|R|}$  and $l_{S}(v) \geqslant l_{R}(v) + \frac{1}{2t}$ for all $v \in B(S)$.
\end{Claim}

We note that Claim \ref{Claim2} is enough to prove Claim \ref{Claim1} and hence our theorem.  Indeed, begin by taking $R = \{ v\}$ for any $v \in V$ and repeatedly replace $R$ with $S$.  After at most $2t^{2}$ such replacements, we will have that $B(R)$ is empty - any bad vertex $v \in B(R)$ would have $w_{R}(v) \geqslant l_{R}(v) \geqslant t$ which is not possible by the definition of $B(R)$.  Moreover, each time we replace $R$ with $S$ we have $|S| \leqslant |R| + t2^{|R|}$ and so our final set will have size bounded above by some function $c_{1}(t)$, as required.

We now describe how to find a suitable set $S$ given some set $R$.  Suppose that $B(R)$ is non-empty.  Let $\mathcal{C}$ be the set
\begin{eqnarray}
\{ C \subseteq R: C = N_{R}(y) \text{ for some } y \in B(R) \} \nonumber
\end{eqnarray}
and label its elements $\mathcal{C} = \{ C_{1}, \ldots, C_{k} \}$.  The set $\mathcal{C}$ is a collection of subsets of $R$ and so $k \leqslant 2^{|R|}$.  For each $C_{i} \in \mathcal{C}$ pick a representative $y_{i} \in B(R)$ such that $ C_{i} =  N_{R}(y_{i}) $.  As $y_{i} \in Y(R)$, we have that $d_{\overline{R}}(y_{i}) < t$ and so, as $d(y_{i}) \geqslant t$, we can pick some $x_{i} \in Y(R)$ such that $y_{i}$ and $x_{i}$ are adjacent. Let $X = \{x_{1}, \ldots, x_{k} \}$ and let
\begin{eqnarray}
S = R \cup X \cup N_{\overline{R}}(X). \nonumber
\end{eqnarray}
Clearly $R \subseteq S$.  Noting that $d_{\overline{R}}(x) \leqslant t-1$ for each $x \in X$, which holds as $X \subseteq Y(R)$, it follows that $|S| \leqslant |R| + tk \leqslant |R| + t2^{|R|}$.  It remains to check that $l_{S}(y) \geqslant l_{R}(y) + 1/2t$ for all $y \in B(S)$.  Recall that for each $v \in V$ we have $f_{S}(v) \geqslant f_{R}(v)$.  Thus, to show that $l_{S}(y) \geqslant l_{R}(y) + 1/2t$ for $y \in B(S)$ it is sufficient to find $v \in N(y)$ with $f_{S}(v) \geqslant f_{R}(v) + 1/2t$.

Given $y \in B(S)$ let $C_{i} \in \mathcal{C}$ be such that $N_{R}(y) = C_{i}$.  We have two cases to deal with depending on whether or not $y$ is adjacent to $x_{i}$.  If $y$ is not adjacent to $x_{i}$ then there are a few further sub cases to deal with.

\medskip
\noindent
\textbf{Case 1:}  $x_{i} \in N(y)$.
\smallskip

If $y$ is adjacent to $x_{i}$ then, as $x_{i} \in Y(R) \cap S$, we have $f_{R}(x_{i}) < 1/2$ while $f_{S}(x_{i}) = 1$ and so we are done.

\medskip
\noindent
\textbf{Case 2:}  $x_{i} \notin N(y)$.
\smallskip

If $y$ is not adjacent to $x_{i}$ then there exists some clique $Z \subseteq V$ of order $p-2$ such that adding an edge between $y$ and $x_{i}$ turns $Z \cup \{ x_{i}, y \}$ into a copy of $K_{p}$.  Recalling that $N_{R}(y)= N_{R}(y_{i})$, we note that $Z \nsubseteq R$ as otherwise $Z \cup \{ x_{i}, y_{i} \}$ would be an example of a copy of $K_{p}$ in $G$.  Thus there exists some $z \in Z \setminus R$ such that $z$ is adjacent to $x_{i}$ and $y$.  We conclude the proof by showing that $f_{S}(z) \geqslant f_{R}(z) + 1/2t$.

\medskip
\noindent
\textbf{Case 2a:}  $z \in \overline{R} \setminus R$.
\smallskip

If $z \in \overline{R} \setminus R$ then $z \in S$ (as it is adjacent to $x_{i}$) and so $f_{S}(z) = 1$ while $f_{R}(z) = 1/2$.

\medskip
\noindent
\textbf{Case 2b:} $z \in Y(R) \cap \overline{S}$.
\smallskip

If $z \in Y(R) \cap \overline{S}$ then $f_{S}(z) \geqslant 1/2$ while $f_{R}(z) \leqslant (t-1)/2t$, as $d_{R}(z) \leqslant t-1$. 

\medskip
\noindent
\textbf{Case 2c:}  $z \in Y(R) \cap Y(S)$.
\smallskip

If $z \in Y(R) \cap Y(S)$ then $f_{R}(z) = d_{R}(z)/2t$ and $f_{S}(z) = d_{S}(z)/2t$.  As $x_{i} \in Y(R) \cap S$ and $R \subseteq S$, we have that $d_{S}(z) \geqslant d_{R}(z) + 1$ and so $f_{S}(z) \geqslant f_{R}(z) + 1/2t$.

\medskip

In all cases, we have show than there is some $v \in N(y)$ with $f_{S}(v) \geqslant f_{R}(v) + 1/2t$.  As a result, we have that $l_{S}(y) \geqslant l_{R}(y) + 1/2t$ for all $y \in B(S)$.  This completes the proof of Claim \ref{Claim2} which in turn proves Claim \ref{Claim1} and hence our theorem.
\end{proof} 

As proved in the introduction, Theorem \ref{Thm1} can be used to show that there exists a constant $c(t,p)$ such that, for $n$ sufficiently large, we have sat$_{t}(n,p) = tn - c(t,p)$.  From a quantitative perspective, Theorem \ref{Thm1} gives an upper bound for $c(t,p)$ that is larger than a tower of exponentials of height $2t^{2}$.  This upper bound can be greatly improved by, in the proof of Theorem \ref{Thm1}, replacing $\mathcal{C}$ with its set of maximal elements (with respect to set inclusion).  Under this change, $\mathcal{C}$ becomes an antichain (meaning that if $A, B \in \mathcal{C}$ then $A \not\subseteq B$) whose elements have size at most $t-1$.  From this, the LYMB-inequality, due to Lubell \cite{LYM-L}, Yamamoto \cite{LYM-Y}, Meshalkin \cite{LYM-M} and Bollob\'{a}s \cite{LYM-B}, shows us that $|\mathcal{C}| \leqslant \binom{|R|}{t-1}$.  As a result, it is possible to prove
\begin{equation}\label{eq2}
c(t,p) \leqslant t^{(t^{(2t^{2})})}. 
\end{equation}

The nature of the proof of Theorem \ref{Thm1} leads us to believe that (\ref{eq2}) is not a good upper bound for $c(t,p)$. For example, the proof only used that $G$ is $K_{p}$-saturated for some $3 \leqslant p \in \mathbb{N}$ and didn't make any use of $p$'s actual value.  Moreover, in the proof of Claim \ref{Claim2} we only used the $K_{p}$-saturated condition on missing edges in $Y(R)$ rather than on all missing edges in $G$.  In Section 3 we construct graphs that give a lower bound for $c(t,p)$.  We believe this lower bound to be closer to the behaviour of $c(t,p)$ than the upper bound (\ref{eq2}) obtained from Theorem \ref{Thm1}.

\section{Constructing $K_{p}$-saturated graphs}
\begin{proof}[Proof of Theorem \ref{Thm2}]

Let $n, t \in \mathbb{N}$ with $t \geqslant4$ and $n \geqslant t \big( 1 +   \binom{t-1}{\lfloor t/2 \rfloor -1} \big)$.  We begin by constructing a graph $G(n,t)$ on $n$ vertices that is $K_{3}$-saturated and has minimum degree $t$.  Let $\mathcal{X} = \{ X \subseteq [t] : 1 \in X, |X| = \lfloor t/2 \rfloor \}$ and label its elements $\mathcal{X} = \{ X_{1}, \ldots, X_{r} \}$.  The vertices of $G(n,t)$ are split into vertex classes $C,H, V_{1},\ldots,V_{r}, W_{1},\ldots,W_{r},$ where 
\vspace{5mm}
\begin{itemize}

  \item $H = \{h_{1}, \ldots, h_{t} \}$,
  \item each $V_{i}$ has $\lfloor t/2 \rfloor$ vertices,
  \item each $W_{i}$ has $\lceil t/2 \rceil$ vertices,
  \item $C$ has the remaining $n - t \big( 1 +   \binom{t-1}{\lfloor t/2 \rfloor -1} \big)$ vertices.

\end{itemize}
The edges of $G(n,t)$ are as follows:
\begin{itemize}

  \item $C$ is fully connected to $H$,
  \item each $V_{i}$ is fully connected to the set $\{ h_{k} : k \in X_{i} \}$,
  \item each $W_{i}$ is fully connected to the set $\{ h_{k} : k \not\in X
  _{i} \}$,
  \item each $V_{i}$ is fully connected to $W_{i}$.
\end{itemize}
See Figure \ref{fig1} for an example of the construction when $t=4$. It is easy to verify that $G(n,t)$ has minimum degree $t$, is $K_{3}$-saturated and has $tn - f(t)$ edges, for some function $f(t) = \Omega(2^{t}t^{3/2})$.  We now use $G(n,t)$ to create $K_{p}$-saturated graphs for $p > 3$.

\begin{figure}
 \centering 
  \includegraphics[scale=0.60]{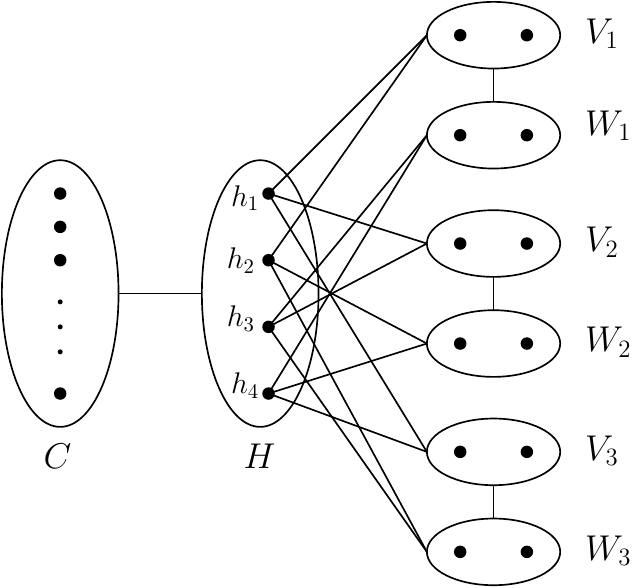}
  \caption{$G(n,4)$ where $X_{1} = \{1,2 \},X_{2} = \{1,3 \},X_{3} = \{1,4 \}.$}\label{fig1}  An edge between two sets (or between a vertex and a set) represents that the two sets (or vertex and set) are fully connected.
\end{figure}

Given a graph $G$, let $G^{*}$ be the graph obtained by adding a new vertex to $G$ and fully connecting it to all other vertices.  If $G$ is a $K_{p}-$saturated graph with minimum degree at least $t$, then $G^{*}$ is a $K_{p+1}-$saturated graph with minimum degree at least $t+1$.  Applying this construction $p-3$ times to the graph $G(n-p+3,t - p + 3)$ (where $t \geqslant p-2$ and $n$ is sufficiently large) gives a $K_{p}$-saturated graphs on $n$ vertices with minimum degree $t$ and fewer than $tn - f(t-(p-3))$ edges.  Thus, for fixed $p$, we have $c(t,p) = \Omega(2^{t}t^{3/2})$.
\end{proof}

The idea of forming a new graph $G^{*}$ from $G$ can also be considered in the other direction.  We say a vertex in a graph is a \textit{conical vertex} if it is connected to all other vertices. Suppose $G$ is a $K_{p}-$saturated graph with minimum degree $t$.  If $G$ has a conical vertex then removing this vertex leaves a $K_{p-1}-$saturated graph with minimum degree $t-1$.  Hajnal \cite{Haj} showed that if $G$ is a $K_{p}$-saturated graph without a conical vertex then $\delta(G) \geqslant 2(p-2)$.  Recall that a consequence of Theorem \ref{Thm1} is that, for $n$ sufficiently large, if $G \in$ Sat$_{t}(n,p)$ then $\delta(G) = t$.  Thus, if $t < 2(p-2)$, these graphs must have a conical vertex and so are of the form $G^{*}$ for some $G \in$ Sat$_{t-1}(n-1,p-1)$.  This leads us to the question:

\begin{Quest}
For which $n, t, p \in \mathbb{N}$ are all graphs in Sat$_{t}(n,p)$ of the form $G^{*}$ for some  $G \in$ Sat$_{(t-1)}(n-1,p-1)$?
\end{Quest}

We remark that there do exist values of $n,t$ and $p$ where Sat$_{t}(n,p)$ contains graphs without a conical vertex.  For example, Sat$_{4}(6,4)$ consists of only the complete tripartite graph $K_{2,2,2}$.  On the other hand Alon, Erd\H{o}s, Holzman and Krivelevich \cite{AEHK} showed that sat$_{4}(n,4) =  4n - 19$ for $n \geqslant 11$, and that  all graphs achieving equality have a conical vertex.  Perhaps it is the case that, for all fixed $t$, all fixed $p \geqslant 4$ and all $n$ sufficiently large, all graphs in Sat$_{t}(n,p)$ have a conical vertex.

\section{Acknowledgements}

I would like to thank Robert Johnson for his valuable advice and for always pointing me in the right direction. I would also like to thank the anonymous referee for their many helpful suggestions on how to improve this paper.  This research was funded by an EPSRC doctoral studentship.

\end{document}